\documentclass[12pt]{amsart}
\usepackage[utf8]{inputenc}

\usepackage[top=2cm, bottom=1cm, left = 2cm, right=2cm]{geometry}

\usepackage{amsmath}
\usepackage{amsfonts}
\usepackage{amssymb}
\usepackage{MnSymbol}
\usepackage{graphicx}
\usepackage{amsthm}
\usepackage[dvipsnames]{xcolor}
\usepackage{mathtools}
\usepackage{wrapfig}
\usepackage{multicol}
\usepackage{enumerate}
\usepackage[font={footnotesize}]{caption}
\usepackage[backref,colorlinks, citecolor={blue}, linkcolor={Red}]{hyperref}

\theoremstyle{plain}
\newtheorem{thm}{Theorem}[section] 

\theoremstyle{definition}
\newtheorem{definition}[thm]{Definition} 
\newtheorem{example}[thm]{Example} 
\newtheorem{lemma}[thm]{Lemma}
\newtheorem{remark}[thm]{Remark}
\newtheorem{proposition}[thm]{Proposition}

\def \re{\rm{Re}}
\def \im{\rm{Im}}

\def \R{{\mathbb R}}
\def \C{{\mathbb C}}
\def \Z{{\mathbb Z}}

\begin{document}

\begin{abstract}
{\footnotesize We provide an integral combinatorial characterization of pseudo-Anosov maps on closed oriented surfaces of genus $g>1$. We show that an orientation-preserving pseudo-Anosov homeomorphism with orientable foliations and fixing all critical trajectories can be encoded as a permutation of $2g+\nu-1$ positive integers, where $\nu$ is the number of singular points of the foliations (disregarding multiplicity). We call such a permutations an \textsc{ordered  block permutation} (OBP), and it satisfies an \textsc{admissiblity} condition. Conversely, we show that a surface along with measured foliations (up to scaling) and the pseudo-Anosov map can be uniquely constructed out of the data of an admissible permutation of $2g+\nu-1$ positive integers. In particular, for closed surfaces, we construct every orientable foliation invariant under a pseudo-Anosov homeomorphism.}
\end{abstract}

\title{Constructing pseudo-Anosov Maps from Permutations and Matrices}
\author{John Hubbard, Ahmad Rafiqi and Tom Schang}

\maketitle

\section{Introduction}
Thurston \cite{Th88} classified mapping classes of compact topological surfaces into three types - periodic, reducible and pseudo-Anosov. $f:S\to S$ (or sometimes its homotopy class $[f]$) is called \textsc{pseudo-Anosov} (pA in short) if there is a transverse pair of singular measured foliations on $S$ invariant under $f$, and a real number $\lambda>1$ called the \textsc{stretch-factor} that the two foliations are respectively stretched and shrunk by, under $f$. The stretch-factor is an important topological invariant of the mapping class, $\log(\lambda)$ being the topological entropy of $f$ (see \cite{Th14}), as well as the length of the geodesic corresponding to the mapping class $[f]$ in the moduli space of complete Riemannian metrics on $S$ under the Teichm\"uller metric (see e.g. the book \cite{H16}). Fried \cite{Fr85} showed that  $\lambda$ is an algebraic unit in the ring of algebraic integers of the number field $\mathbb{Q}(\lambda)$. Moreover, he shows that $\lambda$ is \textsc{bi-Perron}, meaning all its Galois conjugates $\mu_i$ satisfy $\{1/\lambda\leq |\mu_i|\leq\lambda\}$ with at most one conjugate on each boundary component.

It is an open problem to determine precisely what algebraic units occur as pseudo-Anosov stretch-factors, but this is an active area with lots of current research, see for instance \cite{ShStr16, Str17, LeStr20, Pan18, Lo19, Fa23, BRW19, RZ24}. It is this question on the precise nature of the stretch-factors that motivated the current paper. Although many constructions of pseudo-Anosov maps exist in the literature, (e.g. \cite{AY81, P88, V82}, to name just a few), we believe a characterization of pseudo-Anosov maps in terms of \emph{integers} will lend itself more easily to finding the algebraic \emph{integers} that arise as their stretch-factors. 

Lind \cite{L84} showed that all Perron (and therefore biPerron) numbers are eigenvalues of integer \textsc{aperiodic} matrices, i.e. non-negative matrices with some strictly positive power; thus proving a converse to the classical Perron-Frobenius theorem. Using this, Thurston \cite{Th14}, showed that logarithms of \textsc{Perron} numbers (algebraic $\lambda>1$ whose Galois conjugates $\mu_i$ satisfy $|\mu_i|<\lambda$), correspond exactly to the topological entropies of interval maps whose critical points have finite forward orbits.

The present paper builds on work done by one of the authors in a previous paper \cite{BRW16} that constructed orientation preserving or reversing pA maps from certain matrices of $0$'s and $1$'s. However, \cite{BRW16} only constructs a finite number of examples in each genus, each hyperelliptic, and the invariant folitations only have one or two singularities. While the current paper focuses on orientation preserving pA maps with orientable foliations that fix the critical trajectories, it manages to encode all of these, irrespective of the number and multiplicities of the singularities. 

Thurston's classification covers all compact surfaces, however we will assume our surfaces to be orientable and without boundary. Further, we restrict to orientation preserving maps, and since the pseudo-Anosov stretch-factors are our motivation, we may assume that the foliations of the map are orientable. Indeed, if this is not the case, the orientation double cover of the foliation may be taken, one with two points above every non-singular point, each with one of the two orientations of the leaf of the foliation through that point. The map then lifts to this branched double cover and the lifted map has the same stretch-factor. The advantage of doing so is that when the foliations are orientable, the induced isomorphism $f_*$ on  $H_1(S;\Z)$ has as largest eigenvalue the stretch-factor $\lambda$ of $[f]$.

The pair of invariant foliations of a pA map being orientable is equivalent to the foliations being the integral curves of the real and imaginary parts of a holomorphic $1$-form $\omega$ (i.e. an Abelian differential) on a Riemann surface structure on the surface (see \cite[Prop. 5.3.4]{H16}). This implies (but is not equivalent to) each singularity of the foliations having an even number of prongs. The pA homeomorphism fixes the set of zeros of $\omega$, while it multiplies its imaginary and real parts by $\lambda$ and $\lambda^{-1}$ respectively. We will call the foliation whose leaves are scaled by $\lambda$ under $f$ (resp. by $\lambda^{-1}$) as the \textsc{expanding} (resp. \textsc{contracting}) foliation. By a \textsc{critical trajectory} of a foliation we will mean a connected component of the complement of a zero within the singular leaf containing that zero. We will be concerned with orientation-preserving homeomorphisms which fix each critical trajectory of the foliations; we call such pseudo-Anosov homeomorphisms \textsc{oriented-fixed} (OF in short):

\begin{definition} A pseudo-Anosov homeomorphism $f$ is \textsc{oriented fixed} (OF) if its invariant foliations are orientable, and such that all the singularities and all the critical trajectories emanating from the singularities are fixed by $f$.
\end{definition} 

We use a well-known decomposition of the surface carrying a pA map, known as zippered rectangles (cf. \cite{HM79},\cite{V82}), albeit modified to the simpler setting of our OF maps. Given such an $f$, we \emph{choose} a finite segment of a critical trajectory of the contracting foliation, terminating at a singularity. This choice uniquely determines a decomposition of the surface into finitely many rectangles with identifications along the boundaries as we will see below. The decomposition as well as the pA map can then be encoded as a finite permutation $\sigma$ of size $n=2g+\nu-1$, and a list of positive integers $\textbf{k}=(k_1,\dots,k_n)$, where $\nu$ is the number of distinct singularities of the invariant foliations. The pair $(\sigma,\textbf{k})$ defines an \textsc{ordered block permutation} (OBP in short, Def. \ref{OBP}), and it satisfies a condition we call \textsc{admissible} (Def. \ref{admissibledef}). Our main result (cf. Thm. \ref{mainThm}) is then the following.

\begin{thm} A choice of a critical segment of the contracting foliation of an OF pseudo-Anosov map uniquely determines an admissible OBP. Conversely, an admissible OBP uniquely determines a closed oriented surface with a pair of transverse orientable foliations with measures unique up to scale, and an OF pseudo-Anosov map preserving the foliations.  
\end{thm}

We note that while the OBP depends on the choice of a critical trajectory of the contracting foliation and on the length of the segment, different choices of length produce at most $4g-3$ different OBPs. Since there are only finitely many singular trajectories to choose from, we can only obtain a finite number of different OBPs for a given OF pseudo-Anosov map.

The paper is organized as follows: in section \ref{SecRectangular} we recall the zippered rectangle decomposition of the surface carrying a pseudo-Anosov map, modified to our setting. In section \ref{SecCombData} we see how an OF pA map has associated to it a permutation of a list of integers that describe the OBP. In section \ref{SecAdmiss} we study the data of an OBP purely combinatorially and without reference to a surface, and define admissibilty. In section \ref{GeneralShape} we study the general shape of the rectangular decomposition, and in section \ref{SecProof} we provide the construction of the surface and the pA map from an admissible OBP, and prove the main theorem (Thm. \ref{mainThm}). 

\section{Rectangular decomposition}
\label{SecRectangular}

\subsection{Invariant holomorphic $1-$form of a pseudo-Anosov map.} In this section, let $S$ be a compact oriented surface (without boundary) of genus $g\ge2$, and $f:S \to S$ be an orientation-preserving pseudo-Anosov homeomorphism of $S$ with orientable foliations. 

As is well-known from the work of Bers \cite{B78}, and Hubbard-Masur \cite{HM79}, the expanding and contracting foliations of $f$ can be thought of as the integral curves of the imaginary and real parts of a holomorphic $1$-form $\omega\in\Omega^1(X)$, (also known as an \textsc{Abelian differential}), with respect to a Riemann surface structure $[\phi:S\to X]$. The pA map $f$ preserves the zeros of $\omega$, and scales the real and imaginary parts, 
\begin{equation}
	\label{invariantfoliations}
	f^*\,(\re\,\omega)\,=\,\lambda\,\re\,\omega\,\quad {\rm {and}} \quad f^*\,(\im\,\omega)\,=\,\frac{1}{\lambda}\,\im\,\omega
\end{equation}

In a neighborhood $U$ of a point $p\in X$, coordinate charts $z:U\to\C$ can be found such that in the coordinate $z$ on $U$, $\omega_p = z^m\,dz$. The multiplicity $m$ is zero except at a finite number of points of $X$, which correspond to the singularities. At a zero of multiplicity $m$, there are $2(m+1)$ critical trajectories for each of the expanding and contracting foliations. Half of these, $m+1$, are incoming and half outgoing. There is at least one singularity - in fact the sum of the multiplicities of the zeros of the holomorphic $1$-form $\omega$ equals $2g-2 = |\chi(S)|$, where $\chi(S)$ is the Euler characteristic of $S$. 

\subsection{Zippered Rectangles}
\begin{wrapfigure}{r}{0.2\textwidth}
  \begin{center}
  	\vspace{-.6cm}
    \includegraphics[width=0.15\textwidth]{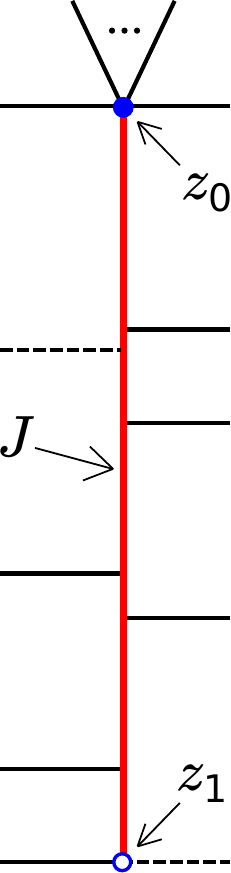}
    \caption{J in red.}
    \label{JNeigh}
  \end{center}
\end{wrapfigure}
We now describe the rectangular decomposition of $X$ associated to a pA map with oriented foliations. This construction appears in a lot of places in the literature, see for instance \cite{HM79}, or Proposition 5.3.4 in \cite{H16}. Choose a segment $\widetilde{J}$ of any positive length of a critical leaf of the contracting foliation terminating at a singularity $z_0\in X$. This is the only choice involved. Now any critical trajectory of the \emph{expanding} foliation, being dense in the surface and transverse to the contracting foliation, eventually intersects $\widetilde{J}$. So lead each critical trajectory of the expanding foliation from each zero of $\omega$ (including $z_0$) until the trajectory intersects $\widetilde{J}$ and then stop. Let $z_1$ be the point where the final expanding critical trajectory intersects $\widetilde{J}$, and let $J=[z_0, z_1]$. Call $\Gamma$ the union of all these critical expanding segments. As is tradition, we will be drawing the expanding foliation horizontally and the contracting foliation vertically, oriented to the right and upwards respectively.
 
We will assume (after possibly reversing the orientation of the horizontal foliation) that the final horizontal trajectory to meet $J$ (at $z_1$) intersects $J$ from the left. Continue this last horizontal trajectory past $z_1$ (the dashed edge in figure \ref{JNeigh}) till it meets $J$ again, and add it to $\Gamma$. We call this dashed edge the \textsc{non-singular edge}, as it contains no singularity, simply going from $J$ back to $J$. All other horizontal egdes drawn from $J$ lead to singularities. We make the assumption that the dashed edge is on the right of $z_1$ solely for brevity of exposition; the other possiblity of it being on the left is very similar, see Remark \ref{rmkrightadmissible} below. Moreover, after having chosen $J$, one can ensure the dashed edge is on the right side of $J$ by reversing the orientation of the expanding foliation.

\begin{figure}[h!]
\includegraphics[scale=.5]{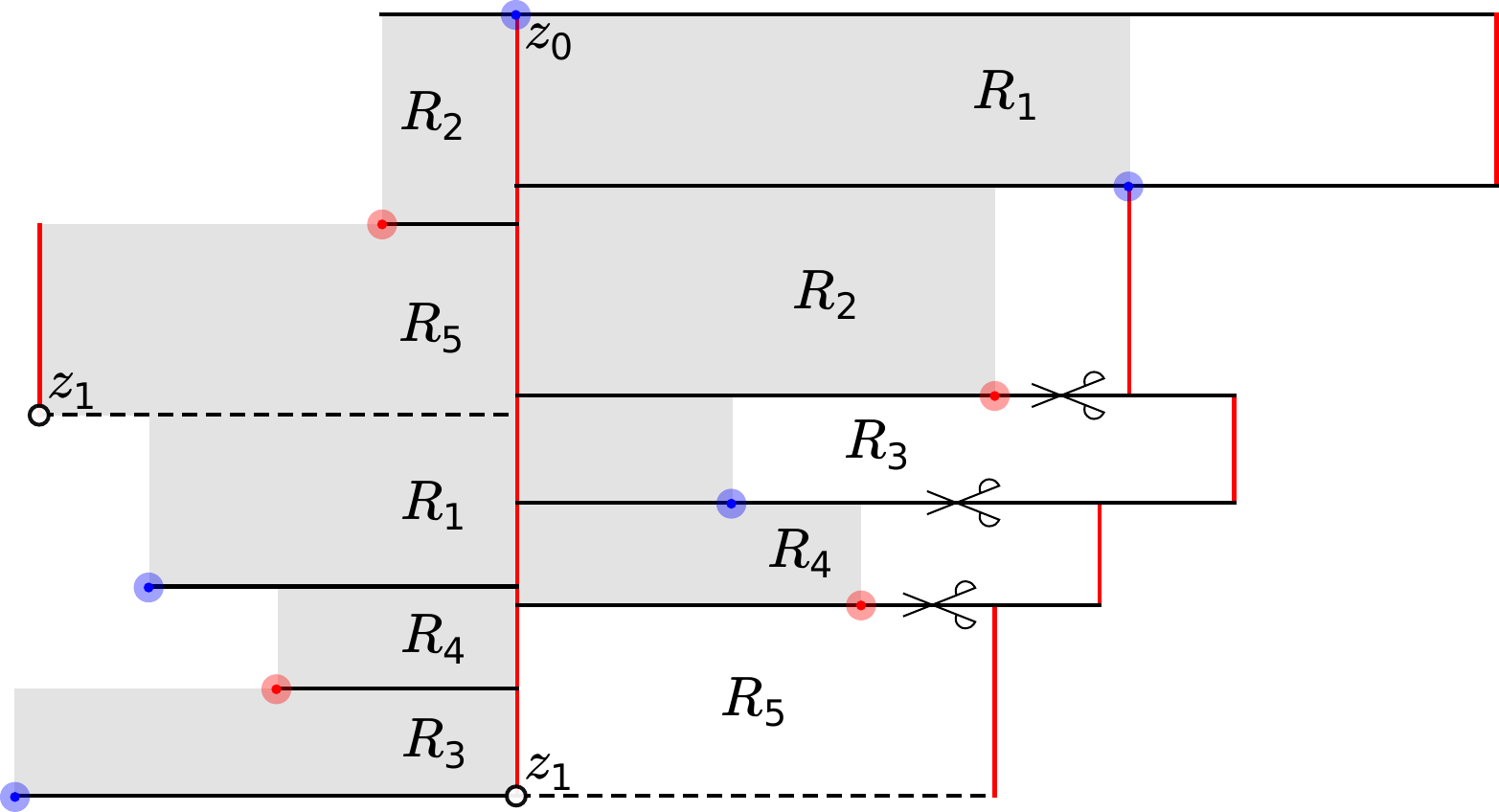}
\caption{ A genus 2 surface with two simple zeros (red and blue) has $2g+\nu-1=5$ rectangles, $R_1$ to $R_5$ drawn on the right of $J$, which cover the whole surface. To the right of each singularity there are slits (indicated by the scissors in the picture). If instead each rectangle is shaded on both the left and right of $J$ until the singularity on its bottom edge, the shaded region also covers the surface \emph{and} contains no slits in its interior. Gluings along the boundary can be determined by the order ($\sigma$) in which the rectangles are attached on the left. }
\end{figure}

The components of the complement of $\Gamma \cup J$ in $X$ are all metric rectangles for the Euclidean metric coming from the $1$-form $\omega$, namely the metric $dx^2+dy^2$ in the coordinate charts where $\omega=dz=dx+i\,dy$. The rectangles alternately have horizontal sides (on $\Gamma$) and vertical sides (on $J$). Let $\mathcal R$ be the set of these rectangles. Each horizontal edge of each rectangle contains a singularity of $\omega$, except the bottom edge of the last rectangle to the right of $J$. 

Number the rectangles in $\mathcal R$ as $R_1, \cdots, R_n$ as they are ordered on the right of $J$. This way $R_1$ is the rectangle with $z_0$ as its top-left corner, and $R_n$ has the non-singular point $z_1$ as its bottom-left corner. The rectangles cover the whole surface, $\bigcup_{i=1}^nR_i=X$, since otherwise the rectangles would form a subsurface whose boundary consists entirely of horizontal leaves, and pA foliations don't have any closed horizontal leaves.

\subsection{The image under $\mathbf{f}$ of the zippered rectangles}
The pseudo-Anosov $f$ stretches each $R_i$ horizontally by the stretch-factor $\lambda$ and shrinks it vertically by the same factor. Since each $R_i$ has its two vertical sides on $J$ and since $f(J)$ is the subinterval $J':=[z_0, f(z_1)]\subset J$, of length $|J'|= |J|/\lambda$, the thinner longer rectangle $f(R_i)$ must begin and end on $J'$. Beginning with its left edge on $J'$, the image $f(R_i)$ thus passes end-to-end within some number of rectangles, before ending with its right vertical edge on $J'$ again. Moreover, if $f(R_i)$ contained one of the drawn critical horizontal segments in its interior, $R_i$ would contain that same critical horizontal segment in its interior as well, since $f$ is a homeromorphism fixing each critical trajectory - but that is impossible since the $R_i$'s are in the complement of the drawn critical trajectories by construction. As a result, the image rectangles $f(R_i),\, i=1,\dots,n$, form thinner rectangular strips, stacked within the rectangles $R_i$ without overlap, and each strip going end-to-end in a single rectangle. We call these image subrectangles strands, (see figure \ref{strandspic}).

\begin{figure}[h!]
\includegraphics[scale=.5]{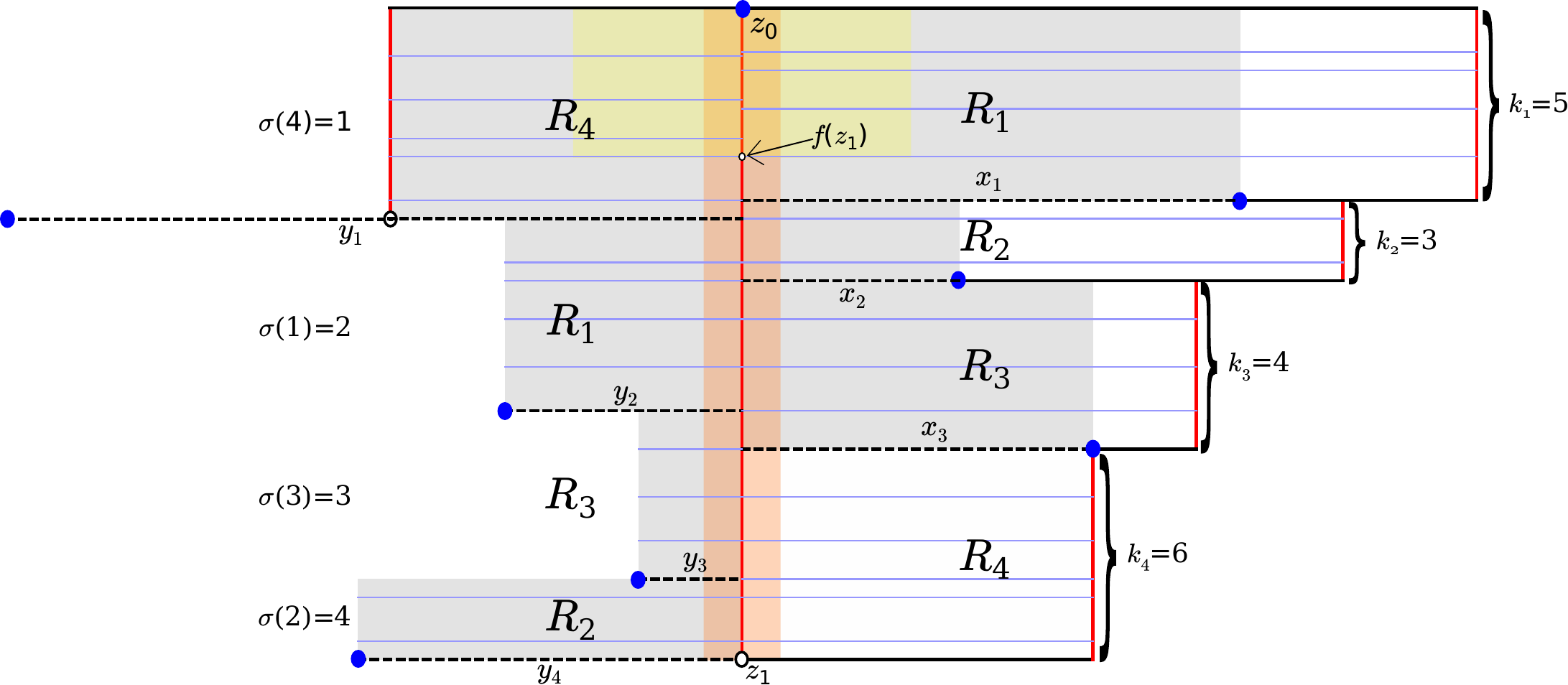}
\caption{A genus $2$ surface with one six-pronged singularity (blue), has $n=4$ rectangles in its decomposition. The highlighted orange neighborhood of $J$ has image the yellow neighborhood of $f(J)$. The blue lines bound the strands: the images of the $R_i$. Counting the strands in each rectangle gives $\textbf{k}=(5,3,4,6)$, so $K=18$, and the order of the $R_i$ on the left defines $\sigma$. Here $\sigma(1,2,3,4) = (2,4,3,1)$.}
\label{strandspic}
\end{figure}

\begin{definition}\label{strands} In the rectangular decomposition of an oriented-fixed map using a critical vertical segment $J$, and critical horizontal segments $\Gamma$ drawn on the surface as above, a \textsc{strand} is a connected component of $X\setminus(J\cup f(\Gamma))$.
\end{definition}

Except for the top-most $n$ strands, the strands to the left and right of the contracting segment $J$ match up exactly. That is, if $K$ is the total number of strands, for each $n<q\leq K$, the right edge of the $q^{th}$ strand to the left of $J$ is identified with the left edge of the $q^{th}$ strand to the right of $J$, (see figure \ref{strandspic}).

We can define an $n\times n$ matrix $A$ of non-negative integers whose $(i,j)^\text{th}$ entry  $a_{ij}$ is the number of times $f(R_j)$ crosses $R_i$. Namely, $a_{ij}$ is the number of strands in $R_i$ that belong to the image $f(R_j)$. The $i^\text{th}$ column of $A$ encodes the image of rectangle $R_i$. Suppose $R_i$ has horizontal length $l_i$ and vertical height $h_i$, measured using the invariant $\omega = dx+\iota\,dy$.

\begin{proposition} 
\label{eigslengthsheights}
The leading eigenvalue $\lambda$ of $A$ is the stretch-factor of $f$.  The vectors 
\[
\text {of lengths  } \quad {\bf l}=\begin{pmatrix}
l_1\\ \vdots\\ l_n\end {pmatrix}\quad  \text{and of heights\ \quad {\bf h}=} \begin{pmatrix}
h_1\\ \vdots\\ h_n\end {pmatrix} 
\]
are eigenvectors of $A^\top$ and $A$ respectively, each with eigenvalue $\lambda$:
\[
A^\top {\bf l}= \lambda {\bf l}\quad  \text{and}\quad A {\bf h} = \lambda {\bf h} .
\]
\end{proposition}

\begin{proof} For a fixed $j$, consider the image of $R_j$, which has length $\lambda l_j$. This image is made up of strands, precisely $a_{ij}$ strands in each $R_i$, which have length $l_i$. Hence the sum $\sum_i a_{ij} l_i$ is exactly the length of $f(R_j)$, i.e., $\lambda l_j= \sum_i a_{ij} l_i$. Thus $\lambda {\bf l}= A^\top{\bf l}$.

Similarly, $\sum_j a_{ij} (h_j/\lambda)$ is the sum of the heights of the $f(R_j)$ crossing $R_i$, which has height $h_i$.  This leads to  $ h_i= \sum_j a_{ij} (h_j/\lambda)$.  Thus $\lambda{\bf h}= A {\bf h}$.
\end{proof}

\subsection{The number of rectangles $\mathbf{n=2g+\nu-1}$}
Suppose the Riemann surface $X$ of genus $g>1$ with an OF pA map $f$ has been decomposed into rectangles $R_1,\cdots,R_n$, after having chosen a critical vertical segment. Suppose the $\nu$ singularities of the foliations have multiplicities $m_1,\dots,m_\nu$, where each $m_i\geq1$. 

By construction, for each $i=1,\dots,(n-1)$, the intersection $R_i\cap R_{i+1}$ contains an incoming critical horizontal trajectory drawn from a point of $J$ to a singularity. As in the figure above, we will label these \textsc{incoming} critical edges as $x_1, x_2, \cdots, x_{n-1}$. Similarly, to the left of $J$, for each $i=1,\cdots,n$, there is a critical horizontal segment drawn at the bottom of the $i^{th}$ rectangle. These emanate from singularities, and terminate at $J$, so we will call these \textsc{outgoing} critical tranjectories, $y_1,\cdots,y_n$. Note that the edge $y_{\sigma(n)}$ below $R_n$ on the left of $J$ is also the edge at the bottom of the figure. 

On the other hand, at a zero of multiplicity $m_i$, there are $m_i+1$ incoming horizontal trajectories. Since the sum of multiplicities is $m_1+\cdots+m_\nu = 2g-2$, we see that \begin{equation}\label{n=2g+nu-1}
n-1=\sum_{i=1}^\nu(m_i+1)=2g-2+\nu,\quad\text{whence,}\quad \fbox{$n=2g+\nu-1$}
\end{equation} 

Note that $g\geq2$ implies $n\geq4$. In fact, since $1\le\nu\le2g-2$, we obtain bounds on $n$ and $g$ in terms of each other 
\begin{equation}\label{ngBounds}\begin{gathered}\boxed{2g\,\le\, n\,\le\,4g-3}\\
\boxed{\frac{n+3}{4}\,\le\, g\,\le\, \frac{n}{2}}
\end{gathered}\end{equation}

\section{The combinatorial data of an Oriented-Fixed pA map}  
\label{SecCombData}

\subsection{The permutation $\mathbf{\sigma}$ and the integers $\mathbf{(}k_1,...,k_n)$  }

We can define the permutation $\sigma$ by saying that the $i^{th}$ rectangle on the right side of $J$, namely $R_i$, is glued in position $\sigma(i)$ on the left side of $J$. Thus, if seen to the left of $J$, the top-most rectangle is $R_{\sigma^{-1}(1)}$, and $R_1$ is glued on the left of $J$ in position $\sigma(1)$.

As discussed above definition \ref{strands}, the image rectangles $f(R_i),\, i=1,\dots,n$, form thinner rectangular strands, stacked within the rectangles $R_i$ without overlap, and each strip going end-to-end in a single rectangle. We then define the positive integers $k_i$ to be the number of strands that pass through the rectangle $R_i$. 

Together, the permutation $\sigma$ and the ordered list of positive integers $\textbf{k}=(k_1,...,k_n)$ consist of the combinatorial data that encodes the complete structure of an oriented-fixed pseudo-Anosov map. It satisfies certain admissibility conditions and, as we will see in section \ref{SecProof}, whenever a pair $(\sigma, \mathbf{k})$ satisfies admissibility, it can be used to construct a unique OF map.

\subsection{Ordered Block Permutations}
\begin{wrapfigure}{r}{0.15\textwidth}
  \begin{center}
   \vspace{-.9cm}
   \includegraphics[width=0.14\textwidth]{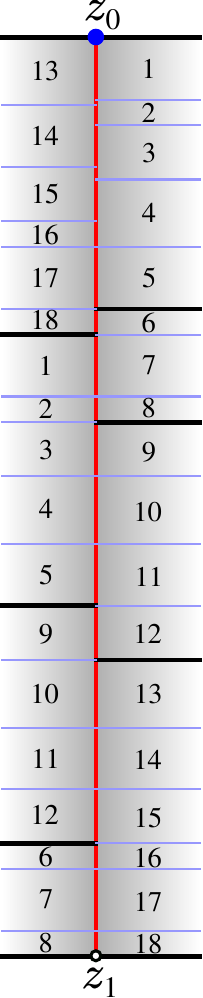}
   \caption{\textcolor{white}{.}}
   \label{strandMap}
  \end{center}
\end{wrapfigure}
To work with the combinatorial data ($\sigma$, $\mathbf{k}$), we number the strands as they are ordered on the right side of $J$, from $1$ to $K:=k_1+\cdots+k_n$. The order in which the strands appear on the left of $J$ defines a permutation $\tau$ of the bigger set $\mathbb{N}_K=\{1,\dots,K\}$. On the right is shown the case for the example in figure \ref{strandspic}, where for instance $\tau(1)=7$ since the $1^{st}$ strand is attached in the $7^{th}$ position on the left of $J$.  This permutation is determined by $\sigma$ and $\textbf{k}=(k_1,\dots,k_n)$, and is what we define below as the ordered block permutation $\tau=\tau_{\sigma,\textbf{k}}$, named so due to its structure of permuting blocks of numbers together. For its definition, we first define the ``\emph{block-number}" function $\beta:(1,\dots,K)\to(1,\dots,n)$, defining $\beta(j)=i$ if the $j^{th}$ strand lies in the $i^{th}$ rectangle $R_i$.

\begin{definition}
\label{OBP}
The \textsc{ordered block permutation} (OBP) $\tau=\tau_{\sigma, {\bf k}}$ is a permutation of the set $\mathbb{N}_K$ defined by
\begin{equation}
\label{OBPEquation}
\hspace{-2cm}\tau_{\sigma, {\bf k}}(j) := \sum_{1\le i<\sigma(\beta(j))}k_{\sigma^{-1}(i)} +j - \sum_{1\le i<\beta(j)}k_i
\end{equation}
\end{definition}

$\tau(j)$ is simply the position in which the $j^{th}$ strand is attached to $J$ on the left. To understand the formula above, note that the $j^{th}$ strand, $1\le j\le K$, is the  $(j - \sum_{1\le i<\beta(j)}k_i)^{th}$ strand of the rectangle $R_{\beta(j)}$. And on the left side of $J$, the number of strands are in the order $(k_{\sigma^{-1}(1)},\cdots,k_{\sigma^{-1}(n)})$, so the first sum on the right in equation \ref{OBPEquation} simply counts how many strands lie above the position where $R_i$ is placed. 

Now, a strand among the first $n$, say strand $j$, with  $1\le j\le n$, is the first strand of the image of $R_j$. Its vertical right side glued to the left side of strand $\tau(j)$. If the $\tau(j)^{th}$ strand isn't one of the first $n$ on the left, then the right side of strand $\tau(j)$ is glued to the left side of strand $\tau(\tau(j))$. This continues until some $p_j>1$ such that $1\le\tau^{\circ p_j}(j)\le n$. The collection of the $p_j$ strands $\{j,...,\tau^{\circ (p_j-1)}(j)\}$ of the $\tau$-orbit of $j$ provide the full image of rectangle $R_j$ under $f$. We call this ordered set of numbers $O_j:= (j,...,\tau^{\circ(p_j-1)}(j))$ the \textsc{orbit} of $j$, where $ 1\leq j\leq n$. Namely the orbit stops as soon as the next image under $\tau$ is back in the subset $\mathbb{N}_n\subset\mathbb{N}_K$. Note the orbits $O_j$ are disjoint, and that for $j\in\mathbb{N}_n$, the first return under iteration of $\tau$ to $\mathbb{N}_n$ is the same as $\sigma$. E.g. in figure \ref{strandMap}, $O_4=(4,10,13)$, so $p_4=3$ and $\tau^{p_4}(4)=1$ same as $\sigma(4)=1$.

\begin{remark}
\label{rmkrightadmissible}
If we instead look to the left of $J$ and flip the horizontal foliation, obtaining the conjugate surface $\bar{X}$, the permutation $\sigma$ and the vector $\mathbf{k}=(k_i)_{i=1}^n$ of positive integers, change to $\sigma^{-1}$ and $(k_{\sigma^{-1}(i)})_{i=1}^n$. The corresponding OBP $\tau_{\sigma^{-1},k_{\sigma^{-1}(*)}}$ is just the inverse $(\tau_{\sigma,\mathbf{k}})^{-1}$. The orbit of $\sigma^{-1}(i)$ continues as the orbit of $i$ backwards, without the $i$, $i$ being it's first return under $\tau^{-1}$. The corresponding matrix is $(A_{\sigma^{-1}(i),\sigma^{-1}(j)})_{i,j=1}^n$, it is also a non-negative matrix, in fact $A_{\tau^{-1}}=P_\sigma AP_{\sigma}^{-1}$ where $P_\sigma$ is the permutation matrix $(\mathbf{e}_{\sigma(1)},\cdots,\mathbf{e}_{\sigma(n)})$, with $P_{\sigma^{-1}}=P_\sigma^{-1}=P_\sigma^\top$.
\end{remark}

\section{Admissible OBPs}\label{SecAdmiss}
\subsection{The combinatorial data}
Here we provide a purely combinatorial way of describing the data of these rectangles, without reference to a surface or map. Let $\sigma$ be a permutation of $\mathbb{N}_n=\{1,2,\cdots,n\}$. Let ${\bf k}=(k_1, \cdots, k_n)\in\mathbb{N}^n$ be $n$ positive integers, and set $K=k_1+\cdots+k_n$.

Partition the ordered set $(1,\dots, K)$ into \textsc{blocks} $B_1, \dots, B_n$, where $B_1$ consists of the ordered set of the first $k_1$ integers $(1,\dots,k_1)$, $B_2$ contains the next $k_2$ elements $(k_1+1,\dots, k_1+k_2)$, and in general for $1\le i\le n$ we have  
\begin{equation}
\label{Block_i}
B_i= (\sum_{j<i}k_j+1,\cdots, \sum_{j\le i}k_j).
\end{equation}
If the strand $j\in(1,\dots,K)$ belongs to block $B_i$, we define $\beta(j)=i$, $\beta:(1,\dots,K)\to(1,\dots,n)$. Using $\beta$ we can define the OBP $\tau$ with equation \ref{OBPEquation} above, (cf. example \ref{firstexample}).

We smiliarly define orbits $O_i$ for each $1\leq i\leq n$ as the ordered set $O_i := (i,...,\tau^{\circ(p_i-1)}(i))$, where $p_i>1$ is the smallest integer such that $\tau^{\circ\,p_i}(i)\in \mathbb{N}_n\subset\mathbb{N}_K$. Here $p_i$ is guaranteed to exist as $\tau$ being a permutation of $\mathbb{N}_K$ has a cycle decomposition. Define $\tau'=\tau'_{\sigma, {\bf k}}:\mathbb{N}_n\to\mathbb{N}_n$ by $\tau'(i)=\tau^{\circ p_i}(i)$, so $\tau'$ is the first return map under iteration of $\tau$ to the subset $(1,...,n)\subset(1,...,K)$.

Define the matrix $A$ by setting the entry $a_{ij}$ in the $i^{th}$ row and the $j^{th}$ column of $A$ equal to the number of elements of the $i^{th}$ block $B_i$ that are in the $j^{th}$ orbit $O_j$. That is, $a_{ij}$ is the number of elements in the intersection: \begin{equation}\label{A}
a_{ij} \quad=\quad |B_i\cap O_j|
\end{equation} 

Entries in the $j^\text{th}$-column of $A$ add up to $p_j:=$  the length of the orbit $O_j$, (cf. the number of strands in the image of $R_j$). And when the union of the orbits is all the strands, $\bigcup_iO_i=\mathbb{N}_K$, then the $i^\text{th}$-row of $A$ adds up to $k_i=$ the size of block $B_i$, (cf. the number of strands that pass $R_i$).

\subsection{Admissibility of the OBP $\tau_{\sigma,\mathbf{k}}$}
With a slight abuse of notation, we will also call the smallest and largest elements of each block $B_i$, namely elements $\sum_{j<i}k_j+1$ and $\sum_{j\le i}k_j$, as the \textsc{top} and \textsc{bottom strands} of $B_i$ respectively.  

\begin{definition} \label{admissibledef}
An ordered block permuation $\tau=\tau_{\sigma,\mathbf{k}}$ is \textsc{admissible} if 
\begin{enumerate}
\item[\textsc{(i)}]
$\tau'=\sigma$;
\item[\textsc{(ii)}]
$\displaystyle\bigcup_{i=1}^n O_i= \{1,\dots,K\}$;
\item[\textsc{(iii)}] Each orbit $O_i$ includes the top and bottom strand of block $B_i$, except the last strand, strand $K\in B_n$, belongs to $O_{\sigma^{-1}(n)}$;
\item[\textsc{(iv)}] The matrix $A$ defined by $a_{ij} = |B_i\cap O_j|$ is irreducible;
\end{enumerate}
\end{definition}

\begin{lemma}
\label{OFpA-OBP}
An ordered block permutation associated to an oriented-fixed pseudo-Anosov homeomorphism is admissible.
\end{lemma}
\begin{proof} Suppose the OBP $\tau_{\sigma, \mathbf{k}}$ is computed from OF pseudo-Anosov homeomorphism as above by choosing some critical vertical segment $J$. Condition \textsc{(i)} expresses the fact that  $f(J)=J'$, so everything you see for blocks along $J$ has to be true of strands along $J'$. Condition \textsc{(ii)} is satisfied since the union $\bigcup_i R_i$ is the whole surface, so is the union $\bigcup_i f(R_i)$. Condition \textsc{(iii)} expresses that all the singularities, and all the leaves emanating from singularities, are fixed by $f$. Thus the singularity on the bottom of $R_i$ must also be on the bottom of $f(R_i)$, so $f(R_i)$ must contain that point as the bottom of some strand, and similarly for the top. Since the last singularity is on the left and $\lambda>1$, $f(R_{\sigma^{-1}(n)})$ continues past $z_1$ and contains the bottom strand of $R_n$.  Condition \textsc{(iv)} is more than satisfied, the matrix $A$ is aperiodic: $A^n$ has strictly positive entries, (see e.g. \cite{Th14}).
\end{proof}

\begin{example}
\label{firstexample}{\footnotesize Here we compute the OBP for the example in figure \ref{strandspic}, with $n=4$, $\sigma = \begin {pmatrix}
 1&2&3&4\\2&4&3&1\end{pmatrix}$ and  $\textbf{k}=(5,3,4,6)$, so $K=18$: The blocks are $B_1=(1,2,3,4,5),\,B_2=(6,7,8),\,B_3=(9,10,11,12),\,B_4=(13,14,15,16,17,18)$, and the permutation $\tau$ is 
\begin{equation}
\tau_{\sigma, {\bf k}} =
\begin{pmatrix}
 1 &  2 & 3 & 4 & 5 & & 6 & 7 & 8 & & 9 & 10 & 11 & 12 & & 13 & 14 & 15 & 16 & 17 & 18\\
 7 & 8 & 9 & 10 & 11 & & 16 & 17 & 18 & & 12 & 13 & 14 & 15 & & 1 & 2 & 3 & 4 & 5 & 6
\end{pmatrix}
\end{equation}
Since $\sigma(4)=1$, the fourth block $B_4=(13,14,15,16,17,18)$ is moved to the first place under $\tau$, so its image is $(1,2,3,4,5,6)$. Continuing this way, as $\sigma(1)=2$, $\tau$ sends the first block $B_1=(1,2,3,4,5)$ to $(7,8,9,10,11)$, and so on.

The orbits of $\tau$ are \begin{align*}
O_1&= ({\color{Green} 1},7, 17, \textcolor{Purple}{5}, 11, 14), [2] & & O_2 = (2, \textcolor{Purple}{8}, \textcolor{Purple}{18}, {\color{Green} 6}, 16) ,[4]\\ 
O_3&= (3,{\color{Green} 9},{\color{Purple} 12},15),[3] & & O_4 = (4,10,{\color{Green} 13}),[1], 
\end{align*}
where we have colored in green the top strands and in purple the bottom strands. Each orbit $O_i$ contains the top and bottom strand of block $B_i$, except $O_n$ doesn't contain the bottom strand of $B_n$, strand $18 = K$, which belongs to $O_{\sigma^{-1}(n)}=O_2$, as required by admissibility.

In square brackets at the end of each orbit is the first return to $(1,2,3,4)$, namely $\tau'$, which is indeed equal to $\sigma$ for this example and the union of the orbits is also all $18$ strands. The matrix $A$ defined as $a_{i,j}=|B_i\cap O_j|$ is \[A=\left(\begin{matrix}
2 & 1 & 1 & 1\\
1 & 2 & 0 & 0\\
1 & 0 & 2 & 1\\
2 & 2 & 1 & 1\\
\end{matrix}\right)\] which is irreducible (in fact aperiodic as $A^2>0$), so combinatorially, $(\sigma, {\bf k})$ is indeed an admissible OBP. The characteristic polynomial of $A$ is $1-7 x+13 x^2-7 x^3+x^4$ and its Perron root $\lambda $ $\sim 4.39026$ equals $\lambda = \left(7+\sqrt{5}+ \sqrt{38+14\sqrt{5}}\right)/4$.}

\qed
\end{example}
 
Next we list some consequences of the combinatorial definition of an admissible OBP, which will be useful in what follows. 

\begin{lemma}\label{lem}
Suppose $\sigma$ is a permutation of $\mathbb{N}_n$ and $\mathbf{k}\in\mathbb{N}^n$ such that the ordered block permutation $\tau_{\sigma,\mathbf{k}}$ is admissible. Further let $A=(a_{ij})$ be the matrix associated to $\tau$, and $\lambda$ the biggest eigenvalue of $A$ in absolute value. Then
\begin{enumerate}[(i)]
\item $k_i\geq2,\,\,\forall i\in\mathbb{N}_n$. On the diagonal, $a_{nn}\geq1$, and the rest $a_{ii}\geq2,\,\,\forall i<n$.
\item $A$ is aperiodic $A^n>0$, and $\lambda>2$.
\item $k_1\ge n$ and $k_{\sigma^{-1}(1)}\ge n$.
\item $\sigma(1)\neq1\neq\sigma^{-1}(1)$ and $\sigma(n)\neq n\neq\sigma^{-1}(n)$. More generally $\nexists\,\,j<n$ such that $\sigma(\{1,\cdots,j\})\subseteq\{1,\cdots,j\}$, and similarly for $\sigma^{-1}$.
\end{enumerate}
\end{lemma}

\begin{proof}
Part (i): Suppose some $k_i=1$. Then the $i^{th}$ row of $A$ consists entirely of $0$s except for a $1$ in the $i^{th}$ position, as $O_i$ must contain the top strand of $B_i$. If $P$ is the matrix obtained by switching the $i^{th}$ and $n^{th}$ rows of the identity matrix $I_n$, we see that  $P^{-1}AP=\left(\begin{smallmatrix} A_{ii} & *\\ \mathbf{0} & 1 \end{smallmatrix}\right)$, where $A_{ii}$ is the matrix formed by deleting the $i^{th}$ row and column of $A$, and $\mathbf{0}$ is a row of $n-1$ zeros. Hence $A$ is reducible, contradicting part \textsc{(iv)} of admissible (\ref{admissibledef}). 

Thus the $k_i\geq2\,\,\forall i$, so the top and bottom strands of the blocks $B_i$ are distinct. Since $\forall i<n$, the orbit $O_i$ contains the top and bottom strands of $B_i$, the entries $a_{ii}=|O_i\cap B_i|\geq2,\,\,\forall i<n$. For $i=n$, we get $a_{nn}\geq1$ since $O_n$ may only contain the top strand of $B_n$.

For (ii), we repeat a well-known proof that a non-negative irreducible matrix with a positive diagonal is aperiodic: Associated to an integer non-negative $n\times n$ matrix $A$ is an oriented graph with $n$ vertices $v_1,\cdots,v_n$, where the number of edges from $v_i$ to $v_j$ is exactly $A_{ji}$. $(A^m)_{ji}$ counts the number of \emph{oriented paths of length $m$} from $v_i$ to $v_j$. Irreducibility of $A$ is equivalent to there \emph{being} an oriented path between any two vertices, while aperiodicity is equivalent to there being an oriented path \emph{of length $n$} from any $v_i$ to $v_j$. Since $\forall i\,\,a_{ii}\neq0$ (part (i)) we know there is a loop at every vertex. Since there are only $n$ vertices, by irreducibility there is a path of length $\leq n$ from any vertex to any other. By adding loops to a path of length less than $n$, we can ensure there is a directed path of length exactly $n$ between any two vertices, hence $A^n>0$.

Since $A$ is aperiodic, by the classical Perron-Frobenius theorem \cite[Ch. XIII, Thms. 2 and 8]{G74}, it has a simple positive eigenvalue $\lambda$ of maximal absolute value. Note that the $k_i$ can not all be equal, otherwise the blocks would line up on the left and right, and orbits would contain top and bottom strands of multiple blocks, contradicting \ref{admissibledef}\textsc{(iii)}. And since the $k_i$ form the row-sums of the matrix $A$, we have $\displaystyle\min_i\{k_i\}<\lambda<\max_i\{k_i\}$ \cite[Ch. XIII, Ineq. 37]{G74}. As $k_i\geq2$, we get $\lambda>2$.  

For (iii), suppose $k_1<n$. Then the $(k_1+1)^{st}$ strand, being one of the first $n$ strands, is the first element of  orbit $O_{k_1+1}$. But it is the top strand of $B_2$, which can only belong to $O_2$. So $k_1=1$, which contradicts part (i) above. To show $n\le k_{\sigma^{-1}(1)}$ we consider the OBP of $(\sigma^{-1},k_{\sigma^{-1}(i)})$ (see Remark \ref{rmkrightadmissible}) which satisfies admissibility, except that in part (\textsc{iii}), $O_n$ contains the bottom strand of $B_{\sigma^{-1}(n)}$.   

For (iv), note that $\sigma(\mathbb{N}_j)\subseteq\mathbb{N}_j$ iff $\sigma(\mathbb{N}_j)=\mathbb{N}_j$ iff $\sigma^{-1}(\mathbb{N}_j)\subseteq\mathbb{N}_j$, since $\sigma$ is a permutation, hence bijective. If $\sigma(\mathbb{N}_j)=\mathbb{N}_j$, by part (iii) all orbits remain within the first $j$ blocks, hence never pass through $B_{j+1}$, contradicting \ref{admissibledef}\textsc{(ii)}. The special cases follow by setting $j=1$ or $j=n-1$.
\end{proof} 

\section{The general shape}\label{SecGeneralShape}

\subsection{Structure of zippered rectangles}
In this section we will study the general structure of the zippered rectangles in our decomposition of an OF pA map. By construction, each rectangle $R_1,\cdots, R_n$ has a singularity on its top edge: for $R_1$ it is at the top-left corner; for $R_{\sigma^{-1}(1)}$ it is at its top-right corner; and for the rest of the $R_i$, it is in the interior of their top edge. Also, each $R_1,\cdots,R_{n-1}$ has a singularity in the interior of its bottom edge - the bottom horizontal edge of $R_n$ doesn't contain a singularity. 

We will refer to the part of the bottom horizontal edge of $R_i$ before (resp. after) the singularity as the \textsc{bottom-left} (resp. \textsc{bottom-right}) edge of $R_i,\,\,\forall 1\leq i<n$. Similarly define \textsc{top-left} and \textsc{top-right} edges of $R_i, \forall 1<i\leq n$. 

\begin{figure}[h!]
\begin{center}
\includegraphics[scale=.27]{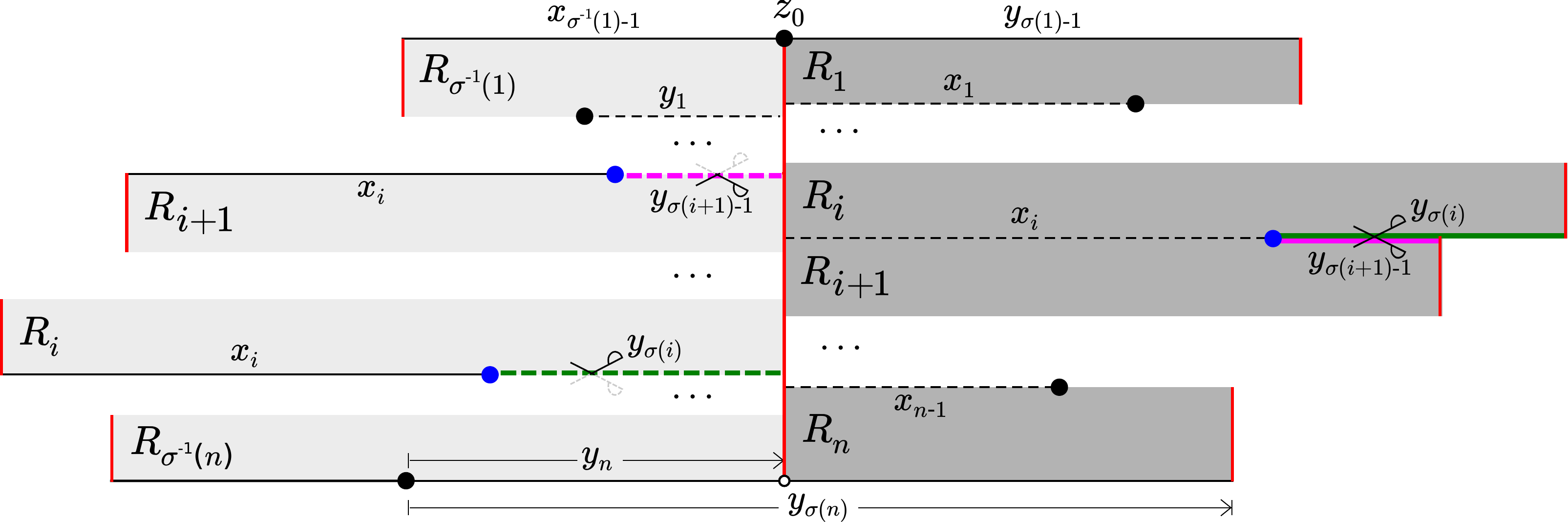}
\caption{\footnotesize In general, $R_1,\cdots, R_n$ are glued along edges $x_1,\cdots,x_{n-1}$ until the singularities. Beyond the singularities, gluings  can be determined by $\sigma$ - the order of $R_i$ on the left of $J$ - along edges $y_1,\cdots,y_n$. Since $R_1$ arrives in position $\sigma(1)$ on the left, its top edge is $y_{\sigma(1)-1}$, well-defined by Lemma \ref{lem}(iv). Similarly, to the left of $z_0$ is the edge $x_{\sigma^{-1}(1)-1}$. Below $R_n$ is the rectangle $R_{\sigma^{-1}(\sigma(n)+1)}$ whose top-right edge $y_{\sigma(n)}$ is glued first to the bottom-right of $R_{\sigma^{-1}(n)}$ along $y_n$, and the rest past $z_1$ to the bottom of $R_n$.}
\label{GeneralShape}
\end{center}
\end{figure}

For each $i\in\mathbb{N}_{n-1}$, the bottom-left of $R_i$ is glued to the top-left of $R_{i+1}$ along the critical edge $x_i$ until the singularity between them. The top edge of $R_{\sigma^{-1}(1)}$, is the edge $x_{\sigma^{-1}(1)-1}$ which is amongst $x_1,\cdots,x_{n-1}$, since {\small $ 2\leq\sigma^{-1}(1)\leq n$} by Lemma \ref{lem}(iv). Similarly, for instance, the bottom-left edge of $R_{\sigma^{-1}(n)}$ is the edge $x_{\sigma^{-1}(n)}$, which is also amongst the $x_i$, since {\small $\sigma^{-1}(n)<n$}. These take care of all the identifications along \emph{incoming} critical horizontal edges.

Similarly to the left of $J$, between the rectangles are the critical edges $y_1,\cdots,y_{n-1}$. In addition $y_n$ is the bottom-right edge of $R_{\sigma^{-1}(n)}$. For each $i\in\mathbb{N}_{n-1}\setminus\{\sigma(n)\}$, the bottom-right of $R_{\sigma^{-1}(i)}$ is glued to the top-right of $R_{\sigma^{-1}(i+1)}$ along $y_i$ from the singularity between them until $J$. 

Among the edges $y_1,\cdots,y_{n-1}$ is the edge $y_{\sigma(n)}$, right below where $R_n$ is attached on the left of $J$, which requires a little attention since $R_n$ doesn't contain a singularity on its bottom edge (cf. figures \ref{Case1}, \ref{Case2}). The edge $y_{\sigma(n)}$ starts from the singularity on the bottom edge of $R_{\sigma^{-1}(n)}$ and continues past the non-singular point $z_1$ till the end of $R_n$. Thus $y_n$ is a subset of $y_{\sigma(n)}$, the latter exceeding it by $l_n$, the length of $R_n$. This takes care of all the identifications along \emph{outgoing} critical horizontal edges. The identifications along the vertical sides all happen on $J$, and are clear. 

Thus, for instance, the bottom-left of $R_1$ is the edge $x_1$ whereas its bottom-right is the edge $y_{\sigma(1)}$. The top edge of $R_1$ is $y_{\sigma(1)-1}$, which is among $y_1,\cdots,y_{n-1}$, again by Lemma \ref{lem}(iv), and it may equal $y_{\sigma(n)}$.

\subsection{Minimality of the permutation $\sigma$.}
When an OBP is defined using an OF pA map, all the singularities of the Abelian differential have multiplicities $m\geq1$. That is, at every singularity, the cone angle is one of $\{4\pi, 6\pi, 8\pi,\cdots\}$. But when an admissible OBP is combinatorially defined, it may be that the cone angle at a ``singularity" is $2\pi$, so the corresponding point is not a singularity, just a fixed-point of $f$. The formula (2) and bounds (3) and (4) no longer hold, so we would like to avoid this possibility.

The multiplicities of the singularities depend entirely on the permutation $\sigma$, so in this section we study the properties of $\sigma$ that guarantee that no singularity has multiplicity $m=0$. Note that apart from the singularities at the very top and bottom, the neighborhood of a singularity on $R_i$ is determined as in the figure \ref{around0}.

\begin{figure}[h!]
\begin{center}
\includegraphics[scale=.24]{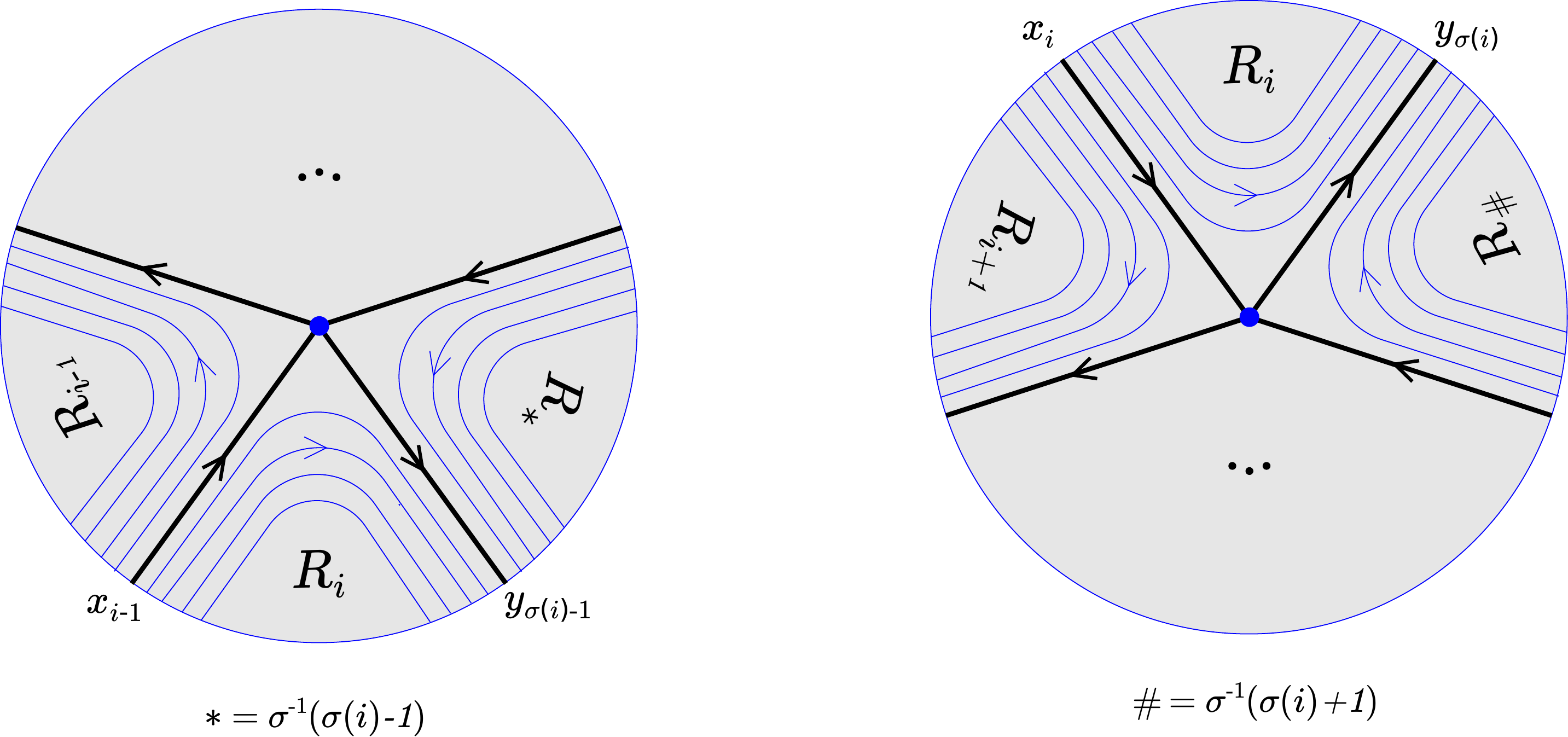}
\caption{\footnotesize On the left is depicted a conformal chart centered at the zero on the \emph{top} edge of $R_i$, for any $i\notin$ {\tiny $\{1,\sigma^{-1}(1),\sigma^{-1}(\sigma(n)+1)\}$}. On the right is the zero on the \emph{bottom} edge of $R_i$, for every $i\notin$ {\tiny $\{\sigma^{-1}(n),\sigma^{-1}(1)-1,\sigma^{-1}(\sigma(1)-1)\}$.} The excluded points correspond to $z_0$ and $R_{\sigma^{-1}(n)}$; the top and bottom.}
\label{around0}
\end{center}
\end{figure}

To determine whether the singularity at the bottom of $R_i$ has cone angle $\geq4\pi$, one can check if $\boxed{\sigma(i+1)\neq \sigma(i)+1}$, as can be seen from figure \ref{around0}. This works for all $i\in\{1,\cdots,n-1\}$, except $i=\sigma^{-1}(n)$ and $i=\sigma^{-1}(1)-1$. These exceptional values correspond to the singularity $z_0$ at the top of $R_1$ and the singularity at the bottom of $R_{\sigma^{-1}(n)}$. There are two cases that occur for these, depending on whether $R_1$ is directly below $R_n$ on the left of $J$ or not, namely whether $R_1$ is or isn't $R_{\sigma^{-1}(\sigma(n)+1)}$. We describe the differences as Case I and Case II below.

5.2.1. \textbf{Case I:} $\mathbf{\sigma(1)\neq\sigma(n)+1.}$ The rectangle $R_1$ is in position $\sigma(1)$ on the left, so its top edge is always $y_{\sigma(1)-1}$. Similarly, below $R_n$ to the left of $J$ is always $R_{\sigma^{-1}(\sigma(n)+1)}$. In case I, as $\sigma^{-1}(\sigma(n)+1)\neq1$,  $R_1$ is not directly below $R_n$. Thus, turning \emph{counter-clockwise} around $z_0$ from $R_1$ leads to $R_{\sigma^{-1}(\sigma(1)-1)}$ which is not $R_n$ (or indeed $R_{\sigma^{-1}(n)}$ either, since then $\sigma(1)=n+1$). On the other hand, turning \emph{clockwise} around $z_0$ leads first into $R_{\sigma^{-1}(1)}$ and then past $x_{\sigma^{-1}(1)-1}$ into $R_{\sigma^{-1}(1)-1}$. The situation is depicted in figures \ref{Case1} and \ref{around01}. 

If $z_0$ were to have cone angle equal to $2\pi$, turning either way around $z_0$ would lead to the same rectangle above, thus $R_{\sigma^{-1}(1)-1}$ would be the same as $R_{\sigma^{-1}(\sigma(1)-1)}$. Thus, in order to test (in case I) whether $z_0$ is a genuine singularity with multiplicity $m\geq1$, we check if $\boxed{\sigma^{-1}(1)-1\neq\sigma^{-1}(\sigma(1)-1)}$. 

\begin{figure}[h!]
\begin{center}
\includegraphics[scale=.25]{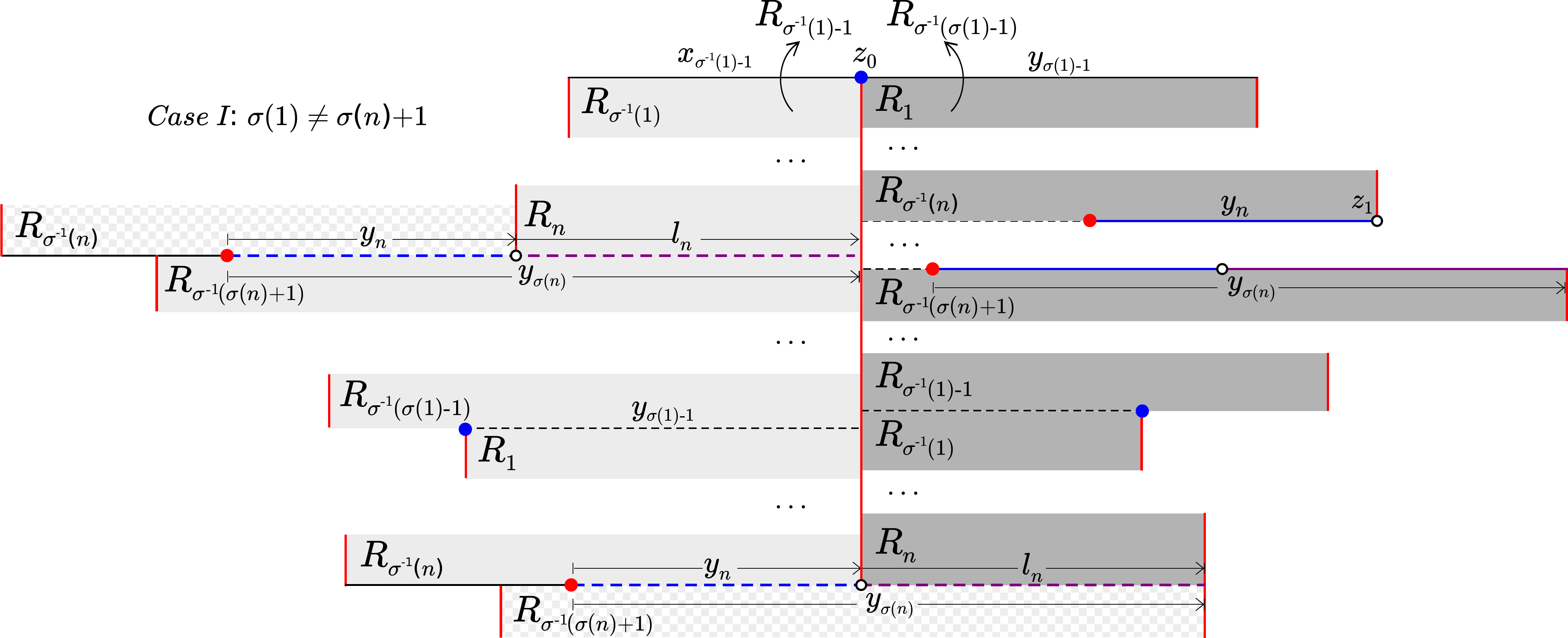}
\caption{\footnotesize Case I: $R_1$ is not directly below $R_n$ to the left of $J$, since $\sigma^{-1}(\sigma(n)+1)\neq1$. Thus the rectangle above $R_1$ is unambiguously $R_{\sigma^{-1}(\sigma(1)-1)}$.}
\label{Case1}
\end{center}
\end{figure}
\bigskip

\begin{figure}[h!]
\begin{center}
\includegraphics[scale=.24]{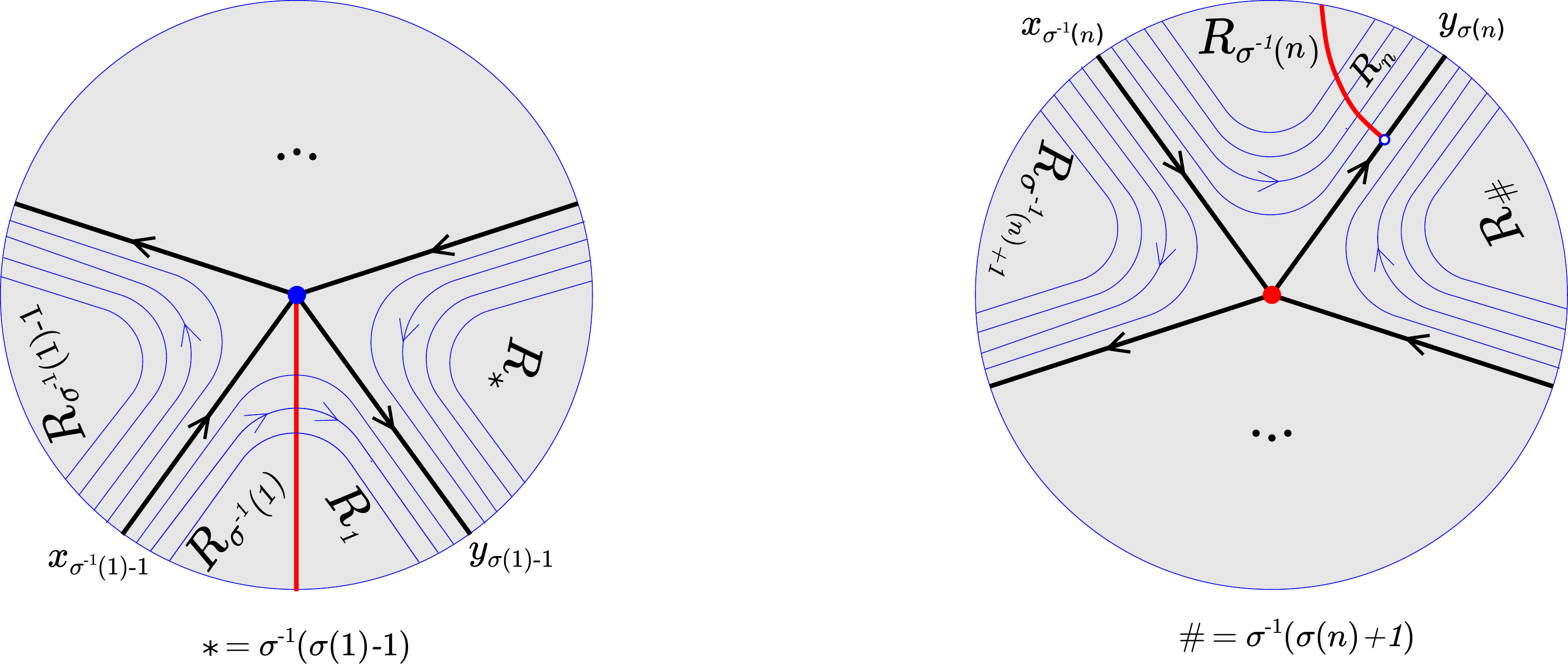}
\caption{\footnotesize Conformal charts centered at the singularity at the top of $R_1$ (left), and at the singularity at the bottom of $R_{\sigma^{-1}(n)}$ (right), are shown for case I.}
\label{around01}
\end{center}
\end{figure}
\bigskip

Similarly, turning clockwise around the singularity at the bottom of $R_{\sigma^{-1}(n)}$ leads to $R_{\sigma^{-1}(\sigma(n)+1)}$ which isn't $R_1$ (or indeed $R_{\sigma^{-1}(1)}$ either). Turning counter-clockwise instead leads to $R_{\sigma^{-1}(n)+1}$. So to test whether the singularity at the bottom of $R_{\sigma^{-1}(n)}$ has cone angle $\geq4\pi$, we check if $\boxed{\sigma^{-1}(n)+1\neq\sigma^{-1}(\sigma(n)+1)}$. 
\bigskip

5.2.2. \textbf{Case II:} $\mathbf{\sigma(1)=\sigma(n)+1.}$ Depicted in figures \ref{Case2} and \ref{around02}, in case II the top and bottom are vertically adjascent, so the two exceptional points of case I coincide. Since, $\sigma(1)=\sigma(n)+1$, we have that the rectangle above $R_1$ is $R_{\sigma^{-1}(\sigma(1)-1)} = R_n$, and indeed $R_n$ is glued above $R_1$ to the left of $J$. However, since $R_n$ doesn't have a singulatiry on its bottom edge, part of the top edge of $R_1$ is also glued to $R_{\sigma^{-1}(n)}$ along $y_n$. As a result, turning counter-clockwise around $z_0$ from $R_1$ doesn't lead to $R_n$, as the formula from case I would suggest, but rather leads into $R_{\sigma^{-1}(n)}$.\\[-1em]

\begin{figure}[h!]
\begin{center}
\includegraphics[scale=.27]{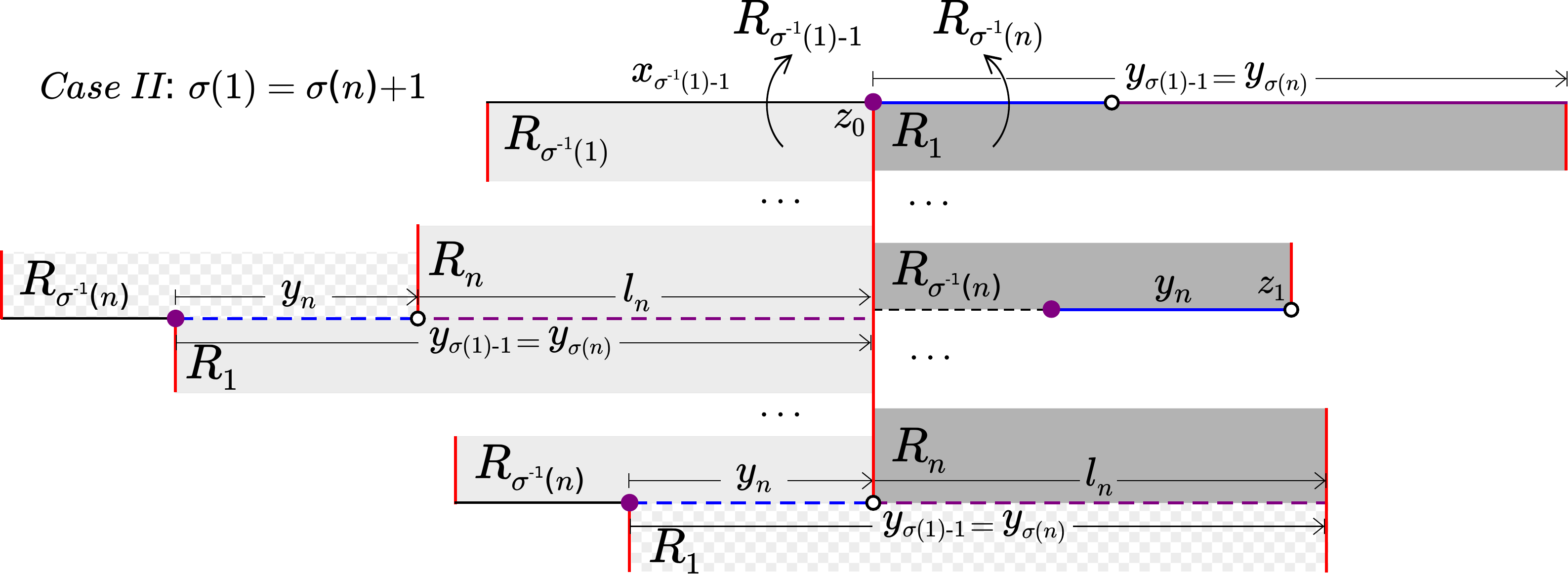}
\caption{\footnotesize Case II: When $R_1$ is below $R_n$ on the left of $J$, so the two exceptional singularities of case I coincide. The top of $R_1$ is first glued to $R_{\sigma^{-1}(n)}$ along $y_n$ and the rest of the top of $R_1$ is glued to the bottom of $R_n$.} 
\label{Case2}
\end{center}
\end{figure}

\begin{figure}[h!]
\begin{center}
\includegraphics[scale=.27]{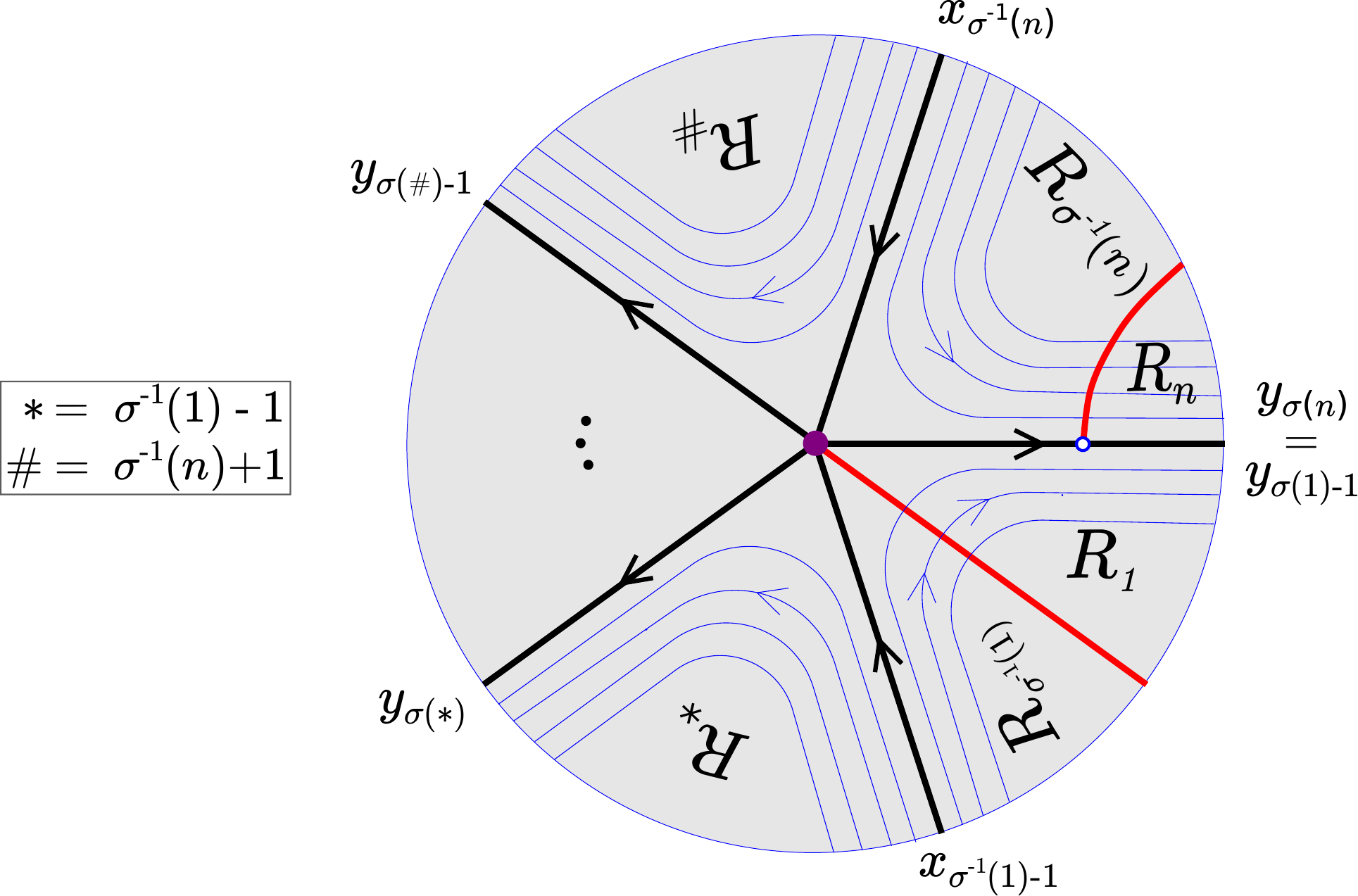}
\caption{\footnotesize Case II: Even though to the left of $J$, the rectangle above $R_1$ is $R_n$, turning counter-clockwise around $z_0$ from $R_1$ leads into the rectangle $R_{\sigma^{-1}(n)}$. }
\label{around02}
\end{center}
\end{figure}

Therefore, in case II, in order to determine if $z_0$ is indeed a singularity of angle $\geq4\pi$, the condition to check is $\boxed{\sigma^{-1}(1)\neq\sigma^{-1}(n)+1}$, since otherwise the incoming edges around $x_{\sigma^{-1}(1)-1}$ and $x_{\sigma^{-1}(n)}$ are the same. 
\bigskip

5.2.3: \textbf{Minimal} $\sigma$. Putting cases I and II together we make the following definition, which guaratees that none of the singularities have angle $2\pi$:

\begin{definition}
A permutation $\sigma$ is called \textsc{minimal} if the following conditions are satisfied:
\begin{enumerate}[(i)]
\item $\sigma(i+1)\neq\sigma(i)+1$, for each $i\in\{1,\cdots,n-1\}\setminus\{\sigma^{-1}(n),\sigma^{-1}(1)-1\}.$
\item When $\sigma(1)=\sigma(n)+1$, $\sigma^{-1}(1)\neq\sigma^{-1}(n)+1$.
\item When $\sigma(1)\neq\sigma(n)+1$, both $\sigma^{-1}(1)-1\neq\sigma^{-1}(\sigma(1)-1)$ and $\sigma^{-1}(n)+1\neq\sigma^{-1}(\sigma(n)+1)$.
\end{enumerate}
\end{definition}

It is clear that when the OBP $\tau_{\sigma, \mathbf{k}}$ is computed for an OF pA map, $\sigma$ is minimal by construction. 

\section{Main Theorem and Proof}\label{SecProof}We can now prove our main result.
\subsection{Main Theorem}
\begin{thm}\label{mainThm} Let $\sigma$ be a permutation of size $n\geq4$, and $\mathbf{k}$ a vector of $n$ positive integers such that $\tau_{\sigma, {\bf k}}$ is an admissible ordered block permutation. Let $A$ be the corresponding $n\times n$ matrix, with leading eigenvalue $\lambda$.  Choose positive $\lambda$-eigenvectors ${\bf l}$ for $A^\top$ and ${\bf h}$ for $A$, and let $R_1, \dots, R_n$ be rectangles with $R_i$ of size $l_i\times h_i$, in fact identified with $[0, l_i]\times [0,h_i] \subset \R^2$.
The rectangles fit together to give a genus $g$ Riemann surface $X$ with a holomorphic $1$-form $\omega$, and the maps 
\[
x+iy \mapsto \lambda x+i \frac y\lambda
\]
in each rectangle fit together to give an oriented-fixed pseudo-Anosov homeomorphism $f_{\sigma,{\bf k}}:X \to X$. When $\sigma$ is minimal, the genus $g=(n+1-\nu)/2$ and is bounded as $\frac{1}{4}(n+3)\le g\le\frac{1}{2}n$, where $\nu$ is the number of distinct singularities. 

Conversely, if $f:X \to X$ is an oriented-fixed pseudo-Anosov homeomorphism of a surface of genus $g$, the choice of a singularity $z_0$ and of a contracting segment $\tilde{J}$ of any length, emanating from $z_0$, defines an admissible ordered block permutation of size $n=2g+\nu-1$, bounded $2g\le n\le4g-3$, (after possibly reversing the orientation of the expanding foliation). 
\end{thm}

\begin{proof}
We have already shown the converse in Lemma \ref{OFpA-OBP}. Given an OF pA map, we may have to reverse the orientation of the expanding foliation, to ensure that the critical expanding segment furthest from $z_0$ along $J$ meets $J$ on the left. Moreover, since we have drawn critical segments from singularities of angle $\geq4\pi$, the permutation $\sigma$ is minimal.\\

To prove the direct statement, assume $\tau_{\sigma, {\bf k}}$ is an admissible OBP, with blocks $B_i$ and orbits $O_j$ computed as in section \ref{SecAdmiss}\textcolor{Red}{.1}. Let $A=(a_{ij})$ be the associated matrix defined by $a_{ij}=|B_i\cap O_j|$. $A$ is then aperiodic and has leading eigenvalue $\lambda>2$ (by Lemma \ref{lem}(ii)), so by the Perron-Frobenius theorem \cite{G74}, $A$ has left and right $\lambda$-eigenvectors $\mathbf{l}$ and $\mathbf{h}$ consisting of entirely positive entries. The eigenvectors $\mathbf{l}$ and $\textbf{h}$ are unique up to scaling by a positive real, so we fix some scaling, but keep it arbitrary so as to construct every pair of orientable foliations carrying a pA map, (not necessarily orientation preserving or fixing the critical trajectories, see remark \ref{AllAbelian}). 

Let $R_i$ be a rectangle of height $h_i$ and length $l_i$, $R_i=[0,l_i]\times[0,h_i]\subset\mathbb{C}\cong\mathbb{R}\bigoplus\iota\mathbb{R}$. Let $J$ be a vertical segment of length equal to the $L_1$ norm of $\mathbf{h}$, i.e. $|J|=\sum_{i=1}^nh_i$.  Arrange the rectangles $R_1,\cdots,R_n$ in order along the right side of $J$ as in Fig. \ref{GeneralShape}. The vertical right sides of the $R_i$ are identified to segments of $J$ according to the permutation $\sigma$. 

The next task is identifications along the $2n$ horizontal edges of the $R_i$. For this we need to uniquely determine the location of the singularities on the edges. With slight abuse of notation we will denote by $x_i$ both the edge itself and its length, and the same for the $y_i$. Thus we need to find the lengths $x_1,\cdots,x_{n-1}$ and $y_1,\cdots,y_n$  such that $y_{\sigma(1)-1}=l_1$, $y_{\sigma(n)}=y_n+l_n$, (cf. Figs. \ref{GeneralShape}, \ref{Case1}, \ref{Case2}), and for each $i$, $1\le i\le n-1$, \begin{equation}\label{x+y=l}
\begin{gathered}
x_i + y_{\sigma(i)} = l_i, \\ 
x_i + y_{\sigma(i+1)-1}  = l_{i+1},
\end{gathered}
\end{equation}
and such that all the terms are positive, except in the second equation for $i=\sigma^{-1}(1)-1$, for then we have $x_{\sigma^{-1}(1)-1} + 0 = l_{\sigma^{-1}(1)}$, (here we have to define $y_0:=0$.)

Moreover, since we are trying to define an \emph{ordered-fixed} map, each edge $x_i$ when stretched by the eigenvalue $\lambda$ of $A$ should yield an edge whose length is equal to $x_i$ plus the sum of the lengths of the rectangles that the $\tau$-orbit $O_i$ passes before the bottom strand of $B_i$; and similarly for the $y_i$. Namely, for each $i<n$, if $O'_i$ is the part of $O_i$ before the bottom strand $\sum_{r=1}^i k_r$ of $B_i$, (guaranteed to be in $O_i$ by Def. \ref{admissibledef}\textsc{(iii)}), and $O''_i$ is the part after the bottom strand, so that \begin{equation}
O_i = O'_i \sqcup \left\lbrace\sum_{r=1}^{i}k_r\right\rbrace\sqcup O''_i,
\end{equation} 
we must have 
\begin{equation}\label{eq_xi_yi}\begin{gathered}
\lambda x_i = \sum_{j\in O'_i}l_{\beta(j)} + x_i,\\
\lambda y_{\sigma(i)} =  y_{\sigma(i)} + \sum_{j\in O''_i}l_{\beta(j)}. 
\end{gathered}\end{equation}
Here $x_i$ is the bottom-left edge of $R_i$ and $y_{\sigma(i)}$ its bottom-right edge. These equations uniquely determine the values of $x_i$ and $y_{\sigma(i)}$ for $1\le i\le n-1$, as
\begin{equation}\label{def_xi_yi}
\boxed{x_i = \frac{1}{\lambda-1}\sum_{j\in O'_i}l_{\beta(j)}},\text{ and }
\boxed{y_{\sigma(i)} =  \frac{1}{\lambda-1}\sum_{j\in O''_i}l_{\beta(j)}}. 
\end{equation}

Note here that for $1\le i\le n-1$, the sets $O'_i$ and $O''_i$ are never empty, as the bottom strand of $B_i$ is never the first or the last element of $O_i$. This is because each orbit has its first element in $B_1$ and its last element in $B_{\sigma^{-1}(1)}$, since $\tau'=\sigma$, and $k_1\ge n$ and $k_{\sigma^{-1}(1)}\ge n$ by Lem.\,\ref{lem}(iii). For $B_1$, the bottom strand $k_1$ isn't the first element of $O_1$; while the last strand of $O_{\sigma^{-1}(1)}$ is the top strand of $B_{\sigma^{-1}(1)}$, hence also not the bottom strand. Thus since the $l_i>0$ and $\lambda>2$, we get $\forall i<n$, \begin{equation}
\boxed{x_i>0} \text{ and } \boxed{y_{\sigma(i)}>0.}
\end{equation}

Moreover, by virtue of $\mathbf{l}$ being the left $\lambda$-eigenvector, we know that $(A^\top{\bf l})_i=\lambda l_i$. And the left side $(A^\top{\bf l})_i=\sum_{j\in O_i}l_{\beta(j)}$ can be decomposed into the orbits before and after the bottom strand of $B_i$, thus giving us \begin{equation}
\sum_{j\in O'_i}l_{\beta(j)}+l_i+\sum_{j\in O''_i}l_{\beta(j)}=\lambda l_i, 
\end{equation}
which in turn simplifies to,\begin{equation}
\underbrace{\frac{1}{\lambda-1}\sum_{j\in O'_i}l_{\beta(j)}}_\text{$x_i$}+\underbrace{\frac{1}{\lambda-1}\sum_{j\in O''_i}l_{\beta(j)}}_\text{$y_{\sigma(i)}$}= l_i.
\end{equation}

Thus by defining $x_i$ and $y_i$ according to Eq. \ref{def_xi_yi}, we also get $\forall i<n$ \begin{equation}
\boxed{x_i+y_{\sigma(i)}=l_i}.
\end{equation}
We have thus uniquely determined the values of $x_1,\cdots,x_{n-1}$ and $y_1,\cdots,\widehat{y_{\sigma(n)}},\cdots,y_n$; the only missing length to determine is that of $y_{\sigma(n)}$.
\bigskip 

We can also decompose each orbit $O_i$ according to when the \emph{top} strand is crossed. Say we denote by $O^{(3)}_{i}$ the first part of the orbit $O_i$ before the top strand of $B_i$, after which the top strand is crossed, and the rest of the orbit as $O^{(4)}_i$. Thus 
\begin{equation}\label{orbits}
O_i = O^{(3)}_i\sqcup\left\lbrace1+\sum_{r=1}^{i-1}k_r\right\rbrace\sqcup O^{(4)}_i,
\end{equation} where $1+\sum_{r=1}^{i-1}k_r$ is the top strand of $B_i$, and in $O_i,\,\forall i\le n$, (\ref{admissibledef}\textsc{(iii)}). However, here $O^{(3)}_1$ and $O^{(4)}_{\sigma^{-1}(1)}$ are empty as the corresponding top strand is the first and last element of the respective orbits. For {\small $i\notin\{1,\sigma^{-1}(1)\}$}, $O^{(3)}_i$ and $O^{(4)}_i$ are all non-empty, for the same reason as above. Using Eq. \ref{orbits}, and the same derivation as for Eq. \ref{def_xi_yi}, we get for each $i$ with $1\le i\le n-1$,
\begin{equation}\label{def_xi_yi_2}
x_i = \frac{1}{\lambda-1}\sum_{j\in O^{(3)}_{i+1}}l_{\beta(j)},\text{\hspace{1cm} and \hspace{1cm}  }
y_{\sigma(i+1)-1} =  \frac{1}{\lambda-1}\sum_{j\in O^{(4)}_{i+1}}l_{\beta(j)},
\end{equation}
where $x_i$ is seen as the top-left edge of $R_{i+1}$, and $y_{\sigma(i+1)-1}$ as its top-right edge. Let us first see why the $x_i$ defined by Eq. \ref{def_xi_yi_2} are the same as defined by Eq. \ref{def_xi_yi} above. 

For each $i<n$, the orbits $O_i=\{i,\tau(i),\cdots,\tau^{\circ(p_i-1)}(i)\}$ and $O_{i+1}=\{i+1,\tau(i+1),\cdots,\tau^{\circ(p_{i+1}-1)}(i+1)\}$ start with adjacent elements $i$ and $i+1$ respectively, both in block $B_1$ (as $k_1\ge n$, Lem. \ref{lem}(iii)). Unless $\tau(i)$ is the bottom strand of $B_i$ (and therefore $\tau(i+1)=\tau(i)+1$ the top strand of $B_{i+1}$), the orbits continue to remain adjacent, and pass through the same block again. This behaviour continues by virtue of the definition of an admissible ordered block permutation, and the orbits $O_i$ and $ O_{i+1}$ only cease being adjacent when the orbit $O_i$ passes the bottom strand of $B_i$, (simultaneously as $O_{i+1}$ passes the top strand of $B_{i+1}$). The lengths $l_{\beta(j)}$ are thus the same for $j\in O'_i$ and $j\in O^{(3)}_{i+1}$, and hence Eqs. \ref{def_xi_yi} and \ref{def_xi_yi_2} define the same values for the $x_i, 1\le i\le n-1.$

Similarly, Eq. \ref{def_xi_yi_2} recalculates the $y_i$ (except $y_{\sigma(n)}$). These values include $y_1,\cdots,\widehat{y_{\sigma(1)-1}},\cdots,y_{n-1}$. So for $j\notin\{n,\sigma(n)\}$, the value $y_j$ has been twice computed, once according to the orbit $O_{\sigma^{-1}(j)}$ after the bottom strand (Eq. \ref{def_xi_yi}), and once according to the orbit $O_{\sigma^{-1}(j+1)}$ after the top strand (Eq. \ref{def_xi_yi_2}, since if $\sigma(i+1)-1=j$, then $O_{i+1}=O_{\sigma^{-1}(j+1)}$). The two orbits $O_{\sigma^{-1}(j)}, O_{\sigma^{-1}(j+1)}$ can be run backwards, starting with the $j^{th}$ and $(j+1)^{st}$ strands of $B_{\sigma^{-1}(1)}$ and, just as above, remain adjacent and within the same blocks, only ceasing to be adjacent when $O_{\sigma^{-1}(j)}$ hits the bottom strand of $B_{\sigma^{-1}(j)}$, (while $O_{\sigma^{-1}(j+1)}$ simultaneously hits the top strand of $B_{\sigma^{-1}(j+1)}$.) Thus again as above, the lengths  $l_{\beta(j)}$ are the same for $j\in O''_{\sigma^{-1}(j)}$ and $j\in O^{(4)}_{\sigma^{-1}(j+1)}$, and hence Eqs. \ref{def_xi_yi} and \ref{def_xi_yi_2} define the same values for the $y_i, 1\le i\le n-1.$

Now, by the same argument as for the $x_i$ above, we see that $\boxed{y_{\sigma(i+1)-1}>0}$, except when {\small $i=\sigma^{-1}(1)-1$}, since $O^{(4)}_{\sigma^{-1}(1)}=\emptyset$. In this case, we get that the top right edge of $R_{\sigma^{-1}(1)}$ is $0$, as it should be since the singularity $z_0$ is at the top-right corner of $R_{\sigma^{-1}(1)}$. And again, as $\mathbf{l}$ is an eigenvector of $A^\top$, we have for $1\le i\le n-1,$ \begin{equation}
(A^\top\mathbf{l})_{i+1} = \sum_{j\in O^{(3)}_{i+1}}l_{\beta(j)}+l_{i+1}+\sum_{j\in O^{(4)}_{i+1}}l_{\beta(j)}=\lambda l_{i+1}.
\end{equation}
Thus we also get $\forall i<n$ \begin{equation}\begin{gathered}
\underbrace{\frac{1}{\lambda-1}\sum_{j\in O^{(3)}_{i+1}}l_{\beta(j)}}_\text{$x_i$}+\underbrace{\frac{1}{\lambda-1}\sum_{j\in O^{(4)}_{i+1}}l_{\beta(j)}}_\text{$y_{\sigma(i+1)-1}$}= l_{i+1},\\
\text{ Or, }\boxed{x_i+y_{\sigma(i+1)-1}=l_{i+1}}
\end{gathered}
\end{equation}

We have thus uniquely determined all the $x_i$ and $y_i$, satisfying Eq. \ref{x+y=l}. The only thing left to check is $y_{\sigma(n)}=y_n+l_n$. For this, we subtract the two equations in \ref{x+y=l} to get\begin{equation}
y_{\sigma(i+1)-1}-y_{\sigma(i)}=l_{i+1}-l_i
\end{equation}

Adding over $i$ yields, \begin{equation}\left(\sum_{i=1}^{n-1}y_i-y_{\sigma(1)-1}\right)-\left(\sum_{i=1}^{n-1}y_i + y_n - y_{\sigma(n)}\right)=\sum_{i=1}^{n-1}(l_{i+1}-l_i)=l_n-l_1.
\end{equation} 
Simplifying and using $y_{\sigma(1)-1}=l_1$, we get: \begin{equation}\label{y_n}\boxed{y_{\sigma(n)}=y_n+l_n}.
\end{equation}
\bigskip

Thus, after a choice of scaling for the left and right $\lambda$-eigenvectors ${\bf l}$ and ${\bf h}$ of $A$, all the $x_i, y_i$ are uniquely determined. By Eqs. \ref{x+y=l} and \ref{y_n}, the rectangles $R_i$ can be glued by translations along segments of their boundaries: $R_i$ and $R_{i+1}$ along $x_i,\,\,\forall 1\le i<n$; and $R_{\sigma^{-1}(i)}$ and $R_{\sigma^{-1}(i+1)}$ along $y_i$, $\forall 1\le i<n, i\ne \sigma(n)$. $y_n$ followed by the non-singular edge (the bottom of $R_n$) are glued to $y_{\sigma(n)}$ (see Fig. \ref{Case1}).  This takes care of all the gluings, and all points have topological disks as neighborhoods, so we obtain a surface $S$ without boundary.

$J$ can be identified with a vertical segment in $\mathbb{C}$ and the interiors of the $R_i$ attached to it on the left and right can be seen as charts to $\mathbb{C}$. A rectangular neighborhood of $J$ of width $\le\min_i\{x_i\}$ can be seen as another chart. Transition maps between small enough neighborhoods of non-singular points are translations ($z\mapsto z+c$), hence conformal. Charts around the singularities can be obtained by applying the map $z\mapsto\sqrt[q]{z}$ in a small neighborhood $U$ of a singularity whose cone angle in the charts $R_i$ is $2\pi q,\, q\ge1$. Since $q\in\mathbb{N}$, $\sqrt[q]{z}$ is conformal in each $U\cap Int(R_i)$, and we obtain a Riemann surface $X$. The holomorphic $1$-form $dz$ in the charts $R_i$ can be pulled back to the surface to obtain an Abelian differential $\omega$.\bigskip

Moreover, we can define the map $f=f_{(\sigma,\mathbf{k})}:X\to X$ as follows. The collection of strands in $O_i$ form a connected long and thin rectangle in the surface. Stretch each rectangle $R_i$ horizontally by $\lambda$, and shrink it vertically by $\lambda$, and place it along this orbit $O_i$. The number of image rectangles $f(R_j)$ thus crossing $R_i$ is $k_i = \sum_{j=1}^na_{i j}$. Each has height $h_j/\lambda$, hence the total height of the rectangles passing through $R_i$ is $\sum_j a_{ij} (h_j/\lambda)$, which is equal to $h_i$ since $\mathbf{h}$ was chosen to be a $\lambda$-eigenvector of $A$. 

Thus the image rectangles fit inside the $R_i$ without overlapping interiors, and $f$ is clearly a homeomorphism in the interior of the $R_i$. The same can be seen in a small neighborhood of an edge $x_i$, since the orbits $O_i$ and $O_{i+1}$ remain adjacent till the bottom strand of $R_i$, and similarly for the $y_i$. The right endpoints of the $x_i$ are the singularities, which by Eq. \ref{eq_xi_yi} are fixed points of $f$, and their multi-sheeted neighborhoods are simply stretched horizontally and shrunk vertically.

Hence $f$ is a homeomorphism. And since the horizontal and vertical foliations of the plane provide orientable singular measured foliations on $X$ invariant under $f$, we have constructed a pseudo-Anosov map that leaves an Abelian differential invariant.

Finally, when $\sigma$ is minimal, the singularities all have cone angle $2\pi q,\,q\ge2$, since otherwise one of the minimality conditions would necessarily be violated. Hence the multiplicities are $m_i\ge1, \forall i\le\nu$, so Eq. \ref{n=2g+nu-1}, applies and so do the bounds \ref{ngBounds}.

\end{proof}
\bigskip

\begin{remark}
\label{AllAbelian}
In this remark, we justify the last sentence of the abstract regarding constructing every orientable foliation on a closed surface that is invariant under a pA map. Suppose $S$ is a closed surface, and $g:S\to S$ a pseudo-Anosov map with orientable invariant foliations $\mathcal{F}^u$ and $\mathcal{F}^s$, but $g$ doesn't necessarily fix all critical trajectories, or is even an orientation-preserving homeomorphism. $g^{\circ 2}$ is orientation preserving, and some higher power $g^{\circ N}$ necessarily fixes all critical trajectories. So $g^{\circ N}$ has an associated OBP. However, taking an iterate of $g$ neither changes the surface nor the invariant foliations. Thus we have constructed the surface with the foliations $\mathcal{F}^u$ and $\mathcal{F}^s$ out of the data of an OBP above. 
\end{remark}

\subsection{Concluding remarks}
The matrices $A$ of size $n=2g+\nu-1$ that one obtains for admissible OBPs as above have the property that for every eigenvalue $\mu$, $1/\mu$ is also an eigenvalue. This can be seen as follows. Define $\gamma_i$ to be the closed curve in the surface running down the middle of $R_i$ with its right end-point connected to its initial point along $J$. Then $\gamma_1,...,\gamma_n$ form a spanning set for the homology group $H_1(X,\mathbb{Z})\cong\mathbb{Z}^{2g}$. Let $S$ be the intersection matrix of the $\gamma_i$. So $S_{ii}=0$ and if $\sigma$ preserves the relative order of $i$ and $j$, then $S_{ij}=S_{ji}=0$. If $\sigma$ reverses their relative order, $S_{ij}$ is $+1$ above the diagonal and $-1$ below. $S$ is the intersection form, expressed in terms of a spanning set (when $\nu>1$) rather than a basis of the homology. Moreover, we see that $(A^\top SA)_{ij}$ equals the number of points in $f(\gamma_i)\cap f(\gamma_j)$ and therefore equals $S_{i\,j}$. So\begin{equation}
A^\top SA=S.
\end{equation} Thus if $\mu$ is an eigenvalue of $A$ with eigenvector $\overrightarrow{w}$, we get \begin{equation}
(S\overrightarrow{w}) =  A^\top SA\overrightarrow{w} = \mu A^\top (S\overrightarrow{w})
\end{equation} so $A^\top$, and therefore $A$, has eigenvalue $1/\mu$. This is a well known property of pseudo-Anosov stretch factors and is related to the characteristic polynomial of the induced action $f_*$ on $H_1$ being palindromic, meaning its coefficients are the same read forwards and backwards.

Given an aperiodic integer matrix $A$ whose characteristic polynomial is palindromic, one can ask whether such a matrix is induced by a pseudo-Anosov map with orientable foliations. The answer to this question is still not known. 

While we cannot answer this, we can provide a finite algorithm to decide whether the given matrix $A$ is the matrix of an admissible OBP:
\begin{enumerate}
\item Let $k_1,\cdots,k_n$ be the row-sums of $A$.
\item For each $\sigma$ in the permutation group of $\mathbb{N}_n$, check if the OBP with $\sigma$ and $\mathbf{k}=(k_1,\cdots,k_n)$ is admissible and generates the same matrix $A$.
\end{enumerate}

Finally, we remark that the Perron root $\lambda$ of the matrix $A$ associated to an admissible OBP $\tau_{\sigma,\mathbf{k}}$ satisfies both 
\begin{equation}
\begin{gathered}
2\le\min_{1\le i\le n}\{k_i\}<\lambda<\max_{1\le i\le n}\{k_i\},\\
\hspace{.6cm}\min_{1\le i\le n}\{p_i\}<\lambda<\max_{1\le i\le n}\{p_i\}.
\end{gathered}
\end{equation}
\bigskip

{\bf \noindent Acknowledgements: } We would like to thank Ana Anusic, Hyungryul Baik, Sylvain Bonnot, Andr\'e de Carvalho, Arcelino do Nascimento, Chenxi Wu and Xuan Zhang for many helpful discussions. We would also like to thank the referee for their helpful comments. The examples were found using code written in C++, and the pictures were drawn using Inkscape and Mathematica. We would also like to thank Ummaya Jan and Hisham Rafiqi for their help in the drawing and coloring of the figures.
\bigskip

\bibliographystyle{abbrv}
\bibliography{DilatationsandOBPs}
\end{document}